\documentclass[12pt]{amsart}

\usepackage[margin=1.2in, centering]{geometry}

\DeclareMathOperator{\acosh}{arccosh}

\usepackage{enumerate}

\usepackage[utf8]{inputenc}
\usepackage{amssymb,amscd,amsmath,enumerate,amsthm}
\usepackage[all]{xy}
\usepackage[utf8]{inputenc}
\usepackage[english]{babel}
\usepackage{tikz}
\usetikzlibrary{positioning,chains,fit,shapes,calc}
\usepackage{mathtools}
\usepackage[colorlinks,linkcolor=blue]{hyperref}
\usepackage{ytableau}

\theoremstyle{plain}
\newtheorem{thm}{Theorem}[section]

\newtheorem{cor}[thm]{Corollary}

\newtheorem{prop}[thm]{Proposition}

\newtheorem{lem}[thm]{Lemma}

\theoremstyle{remark}
\newtheorem{remark}{\bf \quad \itshape  Remark}

\theoremstyle{plain}

\theoremstyle{definition}

\renewcommand{\bar}{\overline}

\newcommand{\R}{{\mathbb{R}}}

\newcommand{\Z}{{\mathbb{Z}}}

\newcommand{\bH}{{\mathbb{H}}}

\newcommand{\SL}{\mathrm{SL}}

\newcommand{\PSL}{\mathrm{PSL}}

\newcommand{\Alt}{{\raise 2pt\hbox{$\scriptstyle\bigwedge$}}}

\definecolor{myblue}{RGB}{80,80,160}
\definecolor{mygreen}{RGB}{80,160,80}
\newdimen\nodeSize
\newdimen\nodeDist
\tikzset{
	position/.style args={#1:#2 from #3}{
		at=(#3.#1), anchor=#1+180, shift=(#1:#2)
	}
}

\allowdisplaybreaks

\title{Geodesic cover of Fuchsian groups}

\author{Zhipeng Lu}

\address{Mathematisches Institut,
	Georg-August Universit\"{a}t G\"{o}ttingen,
	Bunsenstra{\ss}e 3-5,
	D-37073 G\"{o}ttingen,
	Germany}

\email{zhipeng.lu@uni-goettingen.de}


\date{}

\begin{document}
	
	\maketitle
	
	\begin{abstract}
		We study unions of fundamental domains of a Fuchsian group, especially those with hyperbolic plane metric realizing the metric of the corresponding hyperbolic surface. We call these unions the \textit{geodesic covers} of the Fuchsian group or the hyperbolic surface. The paper contributes to showing that finiteness of geodesic covers is basically another characterization of geometrically finiteness. The resolution of geometrically finite case is based on Shimizu's lemma.
		
	\end{abstract}

	\section{Introduction}
	In the work of Lu-Meng \cite{Lu-Meng}, a notion called \textit{geodesic cover} of Fuchsian groups is introduced to deal with the Erd\H{o}s distinct distances problem in hyperbolic surfaces. There we initiate the investigation of quantitative aspects of the notion for special Fuchsian groups including the modular group and standard regular surface groups. In this paper, we conduct an overall qualitative study and give a relatively complete description of geodesic covers for general Fuchsian groups. 
	
	We recall the definition as follows. Let $\Gamma\leq\PSL_2(\R)$ be a Fuchsian group acting on the upper half plane $\bH^2$ by M\"{o}bius transformation, and $F\subset\bH^2$ be a fundamental domain of $\Gamma$ whose interior contains no two points in the same orbit of the M\"{o}bius transformation by $\Gamma$. A subset $\Gamma_0\subset\Gamma$, which we always assume contains the identity, is called a \textit{geodesic cover} of $\Gamma$ corresponding to $F$ if for any $p,q\in F$, we have
	\[\min_{\gamma_1,\gamma_2\in\Gamma}d_{\bH^2}(\gamma_1\cdot p,\gamma_2\cdot q)=\min_{\gamma_1,\gamma_2\in\Gamma_0}d_{\bH^2}(\gamma_1\cdot p,\gamma_2\cdot q),\]
	where $\gamma\cdot p$ denotes the M\"{o}bius transformation and $d_{\bH^2}$ denotes the hyperbolic metric in $\bH^2$. A weaker requirement gives rise to another definition: $\Gamma_1\subset\Gamma$ is a geodesic cover of $\Gamma$ if 
	\[\min_{\gamma\in\Gamma}d_{\bH^2}(p,\gamma\cdot q)=\min_{\gamma\in\Gamma_0}d_{\bH^2}(p,\gamma\cdot q).\]
	We may call the former definition the \textit{first geodesic cover} and the latter the \textit{second geodesic cover}. Clearly if $\Gamma_0$ is a first geodesic cover of $\Gamma$, then $\Gamma_0^{-1}\Gamma_0:=\{\gamma_1^{-1}\gamma_2\mid \gamma_1,\gamma_2\in\Gamma_0\}$ is a second geodesic cover of $\Gamma$. We focus on the second geodesic cover since it appears more convenient for estimation. If $\Gamma$ has finite geodesic covers, we call the smallest  size of (first or second) geodesic covers the (first or second) \textit{geodesic covering number} of $\Gamma$. 

	The definitions arises equivalently in the corresponding hyperbolic surface $S_\Gamma:=\Gamma\backslash\bH^2$ endowed the hyperbolic metric from $\bH^2$. For $p,q\in S_\Gamma$, we can pick two representatives (still denoted by $p,q$) in $F$ to realize their distance by $d_{S_\Gamma}(p,q)=\min_{\gamma\in\Gamma_Y}d_{\bH^2}(p,\gamma\cdot q)$. Thus sometimes we may also refer the same notion to the corresponding hyperbolic surface of a Fuchsian group. Note that although geodesic covers may depend on the choice of fundamental domains, the geodesic covering number only depend on the Fuchsian group. A problem may arise that, since in general two fundamental domains may not be commensurable, we are not sure if any two geodesic covers corresponding to different fundamental domains are commensurable.
	
	In \cite{Lu-Meng}, the geodesic covering number of the modular group $\PSL_2(\Z)$ is precisely estimated to be $4$ in the first definition and $10$ in the second. Also, those of standard regular hyperbolic surfaces of genus $g$ ($\geq 2$) are estimated to establish lower bounds of Erd\H{o}s distinct distances problem in such surfaces, see Theorem 1.2 and Proposition 3.1 of \cite{Lu-Meng}.
	
	In this paper, we extend our scope of estimation and prove
	\begin{thm}\label{thm-geometrically finite}
		Geometrically finite Fuchsian groups have finite geodesic covers.
	\end{thm}
    This is the combination of Theorem \ref{thm-co-finite} and Theorem \ref{thm-second kind}, which addresses Conjecture 3 of \cite{Lu-Meng}. Our proof is based on Shimizu's lemma introduced by Shimizu \cite{Shimizu}.
    
    Similar to the argument for Theorem 1.1 of \cite{Lu-Meng}, this in turn establishes Guth-Katz type lower bound of Erd\H{o}s distinct distances problem in hyperbolic surfaces corresponding to geometrically finite Fuchsian groups.
    \begin{cor}
    	A set of $N$ points in a hyperbolic surface $Y$ corresponding to a geometrically finite Fuchsian group determines $\geq c\frac{N}{K_Y^3log(K_YN)}$
    	distinct distances, where $K_Y$ denotes the geodesic covering number of $Y$ and $c>0$ is some absolute constant.
    \end{cor}
    To the complementary, in section \ref{subsection-infinitely generated} we show that 
     \begin{thm}
     	Infinitely generated Fuchsian groups (without elliptic elements) can not have finite geodesic covers.
     \end{thm}
	 
	It will be interesting to see more precise computational results on geodesic covering numbers of geometrically finite groups and their applications to various geometric and combinatorial problems. The conjectured relation between those numbers and signatures of Fuchsian groups or Fenchel-Nielsen coordinates in the Techm\"{u}ller space in \cite{Lu-Meng} still calls for more studies, to which the current paper may have implications.
	\subsection*{Acknowledgement} The author is supported by Harald Helfgott's Humboldt Professorship.

	\section{Generals on geodesic covers}\label{section-geodesic cover}
	In this section, we give a construction of geodesic covers for general Fuchsian groups using Dirichlet domains, and establish some preliminary results.
    
    \subsection{Construction of geodesic cover} For a general Fuchsian group $\Gamma\leq\PSL_2(\R)$ and any $z_0\in\bH^2$, we define the \textit{Dirichlet domain} for $\Gamma$ \textit{centered} at $z_0$ to be
    \[D_\Gamma(z_0):=\{z\in\bH^2\mid d_{\bH^2}(z,z_0)\leq d_{\bH^2}(z,\gamma\cdot z_0), \forall\gamma\in\Gamma\}.\] 
    Geometrically, $D_\Gamma(z_0)$ is the intersection of hyperbolic half-planes bounded by perpendicular bisectors of the geodesic segments between $z_0$ and $\gamma\cdot z_0$ for all $\gamma\in\Gamma$. If $z_0$ is not fixed by any non-identity element of $\Gamma$, then $D_\Gamma(z_0)$ becomes a connected fundamental domain for $\Gamma$. 
    Otherwise, $z_0$ may only be fixed by an elliptic element of $\Gamma$, in which case $z_0$ is an \textit{elliptic point} of $\Gamma$. In the case, $D_\Gamma(z_0)$ is not a fundamental domain, but we have 
    \begin{lem}\label{lem-Dirichlet domain}
    	For any Fuchsian group $\Gamma$ and an elliptic point $z_0$ of $\Gamma$, $D_\Gamma(z_0)$ meets each orbit $\Gamma\cdot z$ at least $|Stab_{\Gamma}(z_0)|$ times. 
    \end{lem}
    \begin{proof}
    	For any orbit $\Gamma\cdot z$, there is a $\gamma_1\in\Gamma$ such that $d_{\bH^2}(z_0,\gamma_1\cdot z)=d_{\bH^2}(z_0,\Gamma\cdot z)$ by the discrete action of $\Gamma$, hence $\gamma_1\cdot z\in D_\Gamma(z_0)$. Clearly for any $z\in D_\Gamma(z_0)$ and $1\neq\gamma\in Stab_\Gamma(z_0)$, $d(z_0,z)=d(z_0,\gamma\cdot z)$. Hence $\gamma\cdot z\in D_\Gamma(z_0)$ and $\gamma\cdot z\neq z$, otherwise $\gamma=1$. Thus $Stab_\Gamma(z_0)\subset D_{\Gamma}(z_0)\cap\Gamma\cdot z$ which contains at least $|Stab_\Gamma(z_0)|$ points. 
    \end{proof}
    For regular $z$, $Stab_\Gamma(z)=\{1\}$ so that the lemma also applies. 
    For any $z_0\in\bH^2$ not fixed by any non-identity elements of $\Gamma$, define
	$B_\Gamma(z_0):=\bigcup_{z\in D_\Gamma(z_0)}D_\Gamma(z)$, and
	\begin{equation}\label{equation-construction of geodesic cover}U_\Gamma(z_0):=\{\gamma\in\Gamma\mid\gamma\cdot D_{\Gamma}(z_0)\cap B_\Gamma(z_0)\neq\emptyset\}.\end{equation}
	Then we simply have 
	\begin{prop}\label{prop-construction of geodesic cover}
		$U_\Gamma(z_0)$ is a geodesic cover of $\Gamma$ corresponding to $D_\Gamma(z_0)$.
	\end{prop}
    \begin{proof}
	For any $z_1,z_2\in D_\Gamma(z_0)$, by Lemma \ref{lem-Dirichlet domain}, there exists $z_2'\in D_{\Gamma}(z_1)$ such that $z_2'\sim z_2$ in $\Gamma\backslash\bH^2$ and $d_{\bH^2}(z_1,z_2')\leq d_{\bH^2}(z_1,\gamma\cdot z_2)$ for any $\gamma\in\Gamma$. Say $z_2'=\gamma'\cdot z_2$, then $z_2'\in\gamma'\cdot D_{\Gamma}(z_0)\cap B_\Gamma(z_0)\neq\emptyset$. Hence $\gamma'\in U_\Gamma(z_0)$.
	\end{proof}
	
	For $\Gamma$ co-compact, i.e. some fundamental Dirichlet domain $D_\Gamma(z_0)$ is compact, $B_\Gamma(z_0)$ is compact by local finiteness of Dirichlet tessellations (Theorem 3.5.1 of \cite{Katok}). In turn, further by local finiteness, $U_\Gamma(z_0)$ contains only finitely many elements of $\Gamma$, which accounts for Proposition 2.1 of \cite{Lu-Meng}.
	
	\subsection{Geodesic covers and finite index subgroups}\label{subsection-infinitely generated} Now we establish two basic results, one relating geodesic covers and finite index subgroups of Fuchsian groups, and one concerning groups without elliptic elements. Especially the latter one shows that generally infinitely generated Fuchsian groups can not have finite geodesic covers.
	\begin{prop}\label{prop-finite index}
		Let $H\leq\Gamma$ be any finite index subgroup of a Fuchsian group $\Gamma$. If $H$ has finite geodesic covers, then so does $\Gamma$ and the geodesic covering number of $\Gamma$ is no more than $[\Gamma:H]$ times that of $H$.
	\end{prop}
    \begin{proof}
    	 Let $\Gamma=Hg_1\cup\cdot\cup Hg_n$ be the right coset decomposition. If $F$ is a fundamental domain of $\Gamma$, then $g_1\cdot F\cup\cdots\cup g_n\cdot F$ is a fundamental domain of $H$, see Theorem 3.12 of \cite{Katok}. We may choose $F$ such that there is a corresponding finite geodesic cover $H_0\subset H$. For any $z_1\neq z_2\in F$, we have
    	\begin{align*}
    	&\min_{\gamma\in\Gamma}d_{\bH^2}(z_1,\gamma\cdot z_2)=\min_{1\leq i\leq n}\min_{\gamma\in H}d_{\bH^2}(z_1,\gamma\cdot(g_i\cdot z_2))\\
    	=&\min_{1\leq i\leq n}\min_{\gamma\in H_0}d_{\bH^2}(z_1,\gamma g_i\cdot z_2)=\min_{\gamma\in H_0g_1\cup\cdots\cup H_0g_n}d_{\bH^2}(z_1,\gamma\cdot z_2).
    	\end{align*}
    	Hence $H_0g_1\cup\cdots\cup H_0g_n$ is a geodesic cover of $\Gamma$ which implies the lemma.
    \end{proof}
    If we can show that some well-behaved finite index subgroup of a Fuchsian group has finite geodesic cover, then so does the bigger group. We will use this fact to deal with geometrically finite groups in the next section.
    
    Next, we prove the following result based on geometric considerations.
    \begin{prop}\label{prop-geodesic cover including generating set}
    	Let $\Gamma$ be any Fuchsian group without elliptic elements, and $\Gamma_0$ be its geodesic cover say corresponding to a fundamental domain $F$ of $\Gamma$. Then  $\Gamma_0\supset\{\gamma\in\Gamma\mid \gamma\cdot F\cap F\neq\emptyset\}:=\Gamma(F)$. 
    \end{prop}
    \begin{proof}
    	Our proof is motivated by the proof of Theorem 9.3.3 of Beardon \cite{Beardon}. For any $1\neq\gamma\in\Gamma(F)$, suppose $w\in\gamma\cdot F\cap F$. Then there is an open disc $N$ with center $w$ and elements $\gamma_0$ ($=1$), $\gamma_1,\dots,\gamma_t\in\Gamma$ such that $\gamma_j=\gamma$ for some $j$, and 
    	\[w\in\bigcap_{i=0}^{t}\gamma_i\cdot F,\quad N\subset\bigcup_{i=0}^t\gamma_i\cdot F.\]
    	$N$ can be chosen small enough so that no other vertices of any $\gamma_i\cdot F$ falls inside $N$. More features are shown in Figure 9.3.1 of \cite{Beardon} on page 220.
    	
    	Now that $\Gamma$ has no elliptic elements, for $N$ small enough, any two points in $N$ are not in the same orbit of $\Gamma$ by local finiteness. Choose any point $z\in \gamma\cdot F\cap N$, and its conjugate $z'=\gamma^{-1}\cdot z\in F$. By definition, there is some $\gamma_0\in\Gamma_0$ such that 
    	\[d_{\bH^2}(w,z)=\min_{\gamma\in\Gamma}d_{\bH^2}(w,\gamma\cdot z')=d_{\bH^2}(w,\gamma_0\cdot z').\]
    	Then $z=\gamma_0\cdot z'=\gamma_0\gamma^{-1}\cdot z$, which implies $\gamma=\gamma_0$ since there are no elliptic elements in $\Gamma$. This shows that any elements of $\Gamma(F)$ belongs to $\Gamma_0$.
    \end{proof}
    \begin{remark}
    	Proposition 3.1 of \cite{Lu-Meng} says the smaller set  $\{\gamma\in\Gamma\mid\gamma\cdot D_{\Gamma}(z_0)\cap D_\Gamma(z_0)\neq\emptyset\}$ is a geodesic cover for $\Gamma=\PSL_2(\Z)$. It is hard to tell if this set consists in a geodesic cover in general from the above proof. 
    \end{remark} 
    The proposition reveals the following basic fact:
    \begin{cor}\label{cor-infinite geodesic cover}
    Any infinitely generated Fuchsian group without elliptic elements can not have finite geodesic covers.	
    \end{cor}
    \begin{proof}
    	Note that in general, $\Gamma(F)$ generates $\Gamma$, see Theorem 9.2.7 of \cite{Beardon}. With the notations in Proposition \ref{prop-geodesic cover including generating set}, if some $\Gamma_0$ is finite, so is the corresponding $\Gamma(F)$, then $\Gamma=\langle\Gamma(F)\rangle$ becomes finitely generated. 
    \end{proof}
    Note that in general any Fuchsian group, has a finite index subgroup without elliptic elements, which in the case of geometrically finite groups is accounted as Nielsen-Fenchel-Fox theorem, see Proposition \ref{prop-Fox} later. Thus the corollary exhibits groups with only infinite geodesic covers for general infinitely generated Fuchsian groups. To conclude the scenario, we will show that finitely generated Fuchsian groups have finite geodesic covers in the next section.
    \section{Geodesic covers of geometrically finite Fuchsian groups} 
    \subsection{Geometry of geometrically finite Fuchsian groups} This section contributes to introducing preliminary notions on geometrically finite Fuchsian groups following Katok \cite{Katok}. 
    
    Call a Fuchsian group $\Gamma\leq\PSL_2(\R)$ \textit{geometrically finite} if there exists one (hence every) geodesically convex fundamental domain of $\Gamma$ having finitely many sides. Equivalently, $\Gamma$ is finitely generated. 
    Also co-finite Fuchsian groups, i.e. with fundamental domains of finite hyperbolic area, are precisely geometrically finite of first kind.

    We describe some basic facts on the geometry of boundary of a fundamental Dirichlet domain $D:=D_\Gamma(z_0)$ for an arbitrary Fuchsian group $\Gamma$. 
    By definition, the boundary of $D$ is a collection of geodesic segments or segments of the real axis (free sides), called \textit{sides}. An intersection point of two sides is called a \textit{vertex}. Two points of $D$ are congruent, i.e. in the same orbit of $\Gamma$, if and only if they belong to the boundary $\partial D$. In this case, they are of same distance to $z_0$ (see Theorem 9.4.3 of \cite{Beardon}). 
    
    Let us consider the vertices in congruence, each class of which is called a \textit{cycle}. By local finiteness, a cycle contains only finitely many points. If a vertex $v\in\bH^2$ of $D$ is an elliptic point fixed by $\gamma\Gamma$, it must lie on $\partial D$, and every elliptic point of $\Gamma$ is congruent to one on $\partial D$. Also the \textit{elliptic cycle} of $v$ is fixed by elliptic elements congruent to $\gamma$ in $\Gamma$. This reveals the following
    \begin{prop}[Theorem 3.5.2 of \cite{Katok}]\label{prop-elliptic cycles}
    	With the above notations,  the elliptic cycles of $\partial D$ 1-1 correspond to the congruence classes of elliptic points or maximal elliptic subgroups of $\Gamma$. 
    \end{prop}
    
    If $v$ is an elliptic vertex of $D$, fixed by $\gamma\in\Gamma$ say with order $k$, then $\gamma$ maps edges to edges and the inner angle at $v$ is at most $2\pi/k$. If $k=2$, then $\gamma$ flips the edge containing $v$ around $v$. Moreover, we have
    \begin{prop}[Theorem 3.5.3 of \cite{Katok}]\label{prop-angle sum in elliptic cycle}
    	With the above notations, let $\theta_1,\dots,\theta_t$ be the internal angles at all vertices in an elliptic cycle and $m$ be the order of the elliptic elements fixing these vertices. Then $\theta_1+\cdots+\theta_t=2\pi/m$.
    \end{prop}
    If a cycle has no fixed points, we may set $m=1$ and then $\theta_1+\cdots+\theta_t=2\pi$. Regarded as infinite order elliptic elements, each parabolic element of $\Gamma$ has a unique fixed point on $\hat{\R}:=\R\cup\{\infty\}$, which are usually called \textit{cusps} of $\Gamma$. For co-finite groups, we have
    \begin{prop}[Theorem 4.2.5 of \cite{Katok}]\label{prop-cusp}
    	Suppose $\Gamma$ has a non-compact Dirichlet fundamental domain $D$ of finite hyperbolic area. Then
    	
    	$(i)$ each vertex of $D$ on $\hat{\R}$ is a cusp;
    	
    	$(ii)$ each cusp is congruent to a vertex of $D$ on $\hat{\R}$.
    \end{prop}
    Similarly, congruence classes of cusps of $\Gamma$ 1-1 correspond to maximal parabolic subgroups of $\Gamma$. 
    
    Note that co-finite Fuchsian groups are geometrically finite of first kind. As to the remaining geometrically finite groups $\Gamma$ of second kind, according to their limit sets $\Lambda(\Gamma):=\{z\in\hat{\R}\mid z\in \bar{\Gamma\cdot a}\text{ for some }a\in\bH^2\}$, we classify them as follows. If $|\Lambda(\Gamma)|\leq 2$, then $\Gamma$ is \textit{elementary} (Exercise 3.8 of \cite{Katok}), i.e. has finite orbits in $\bH^2\cup\hat{\R}$, which can be described as follows.
    \begin{prop}[Theorem 2.4.3 of \cite{Katok}]\label{prop-elementary groups}
    	Any elementary Fuchsian group is either cyclic or is conjugate in $\PSL_2(\R)$ to a group generated by $g_k=\begin{pmatrix}
    	k&0\\0&1/k
    	\end{pmatrix}$ ($k>1$) and $S=\begin{pmatrix}
    	0&-1\\1&0
    	\end{pmatrix}$.
    \end{prop}
    If $\Gamma$ is finite cyclic, then it is elliptic and $\Lambda(\Gamma)=\emptyset$; if $\Gamma$ is infinite cyclic, then it is parabolic and $|\Lambda(\Gamma)|=1$, or it is hyperbolic congruent to $\langle g_k\rangle$ for some $k>0$, in which case $|\Lambda(\Gamma)|=2$. For $\Gamma$ conjugate to $\left\langle g_k,S\right\rangle$, similarly $|\Lambda(\Gamma)|=2$. For example, $\Lambda\left(\left\langle g_k,S\right\rangle\right)=\{0,\infty\}$ for any $k>1$.
    
    If $|\Lambda(\Gamma)|>2$, then $\Gamma$ is non-elementary and $\Lambda(\Gamma)$ is a perfect nowhere dense subset of $\hat{\R}$ (hence uncountably infinite) by Theorem 3.4.6 of \cite{Katok}. In this case, the complement of $\Lambda(\Gamma)$ in $\hat{\R}$ is a union of countably many open intervals $\{I_j\}_{j=1}^\infty$ that are mutually disjoint. Let $L_j$ be the geodesic striding over $I_j$ and connecting its two end points in $\bH^2$, and let $H_j$ be the open half-plane bounded by $L_j$ away from $I_j$. Now we introduce the notion of \textit{Nielsen region} following 8.4 of \cite{Beardon} as
    \begin{equation}\label{equation-Nielsen region} N_\Gamma:=\bigcap_{j=1}^\infty H_j.\end{equation}    Note that $\{I_j\}$ is $\Gamma$-invariant, so is $\{H_j\}$. Therefore $N_\Gamma$ is a $\Gamma$-invariant geodesically convex subset of $\bH^2$. Actually, it is the smallest such non-empty set, see Theorem 8.5.2 of \cite{Beardon}. The following fact reveals its significance: 
    \begin{prop}[Theorem 10.1.2 of \cite{Beardon}]\label{prop-Nielsen region}
    	A non-elementary Fuchsian group $\Gamma$ is finitely generated if and only if for any convex fundamental domain $D$ of $\Gamma$, the hyperbolic area of $D\cap N_\Gamma$ is finite.
    \end{prop}
    One key observation is that the free sides (segments of $\R$) are each contained in some $I_j$, then $D\cap N_\Gamma$ is an polygon in $\bH^2$ which has finite area by Gauss-Bonnet formula, see the proof on page 254 of \cite{Beardon}.    
   
    \subsection{Geodesic cover of co-finite groups}
    Now we try to establish the seemingly surprising fact that, the construction of (\ref{equation-construction of geodesic cover}) gives a finite geodesic cover for geometrically finite Fuchsian groups of first kind, i.e. co-finite groups, with a technicality assumption that they contain no elliptic elements.
    
    To proceed we introduce two useful corollaries of Shimizu's lemma which first appeared in \cite{Shimizu}. For a given Fuchsian group $\Gamma$, denote by $\tilde{\bH}^2$ the union of $\bH^2$ and the cusps of $\Gamma$.
    \begin{prop}[Lemma 1.26 of Shimura \cite{Shimura}]\label{prop-Shimura}
    	For every cusp $s$ of a Fuchsian group $\Gamma$, there exists a neighborhood $U\subset\tilde{\bH}^2$ of $s$ such that $Stab_\Gamma(s)=\{\gamma\in\Gamma\mid \gamma\cdot U\cap U\neq\emptyset\}$.
    \end{prop}
    \begin{prop}[Theorem 9.2.8 (ii) of \cite{Beardon}]\label{prop-horocyclic region}
	Let $D$ be any locally finite fundamental domain for a Fuchsian group $\Gamma$, $\gamma\in\Gamma$ be a parabolic element, and $U$ be a horocyclic region with $\gamma\cdot U=U$. Then $D$ meets some finitely many distinct translates $\gamma'\cdot U, \gamma'\in\Gamma$.
    \end{prop}
    Here a \textit{horocyclic region} of a parabolic element $\gamma$ is a neighborhood of the fixed point of $\gamma$ congruent to $\{z\in\bH^2\mid Im(z)>t\}$ for some $t>0$. Beardon's proof uses J\o{}rgensen's inequality, which may be seen as a generalized version of Shimizu's lemma. 
    
    Now we modify the construction of (\ref{equation-construction of geodesic cover}) using a \textit{horocyclic surgery} on Dirichlet domains resorting to the above two results. Assume $\Gamma$ is a co-finite Fuchsian group without elliptic elements (with cusps otherwise co-compact), so that for any $z\in\bH^2$, $D_\Gamma(z)$ is a fundamental Dirichlet domain. Fix a point $z_0\in\bH^2$ and let $D:=D_\Gamma(z_0)$. Suppose $s_1,\dots,s_t$ are all the cusp vertices on $\partial D$. According to Proposition \ref{prop-Shimura} and \ref{prop-horocyclic region}, we may choose horocyclic neighborhood $U_i$ of each $s_i$ such that
    \begin{equation}\label{equation-horocyclic neighborhood}
    Stab_{\Gamma}(s_i)=\{\gamma\in\Gamma\mid \gamma\cdot U_i\cap U_i\neq\emptyset\}, \forall i=1,\dots,t,
    \end{equation}
    and $D$ meets only finitely many translates of $U_i$'s. Now define
    \begin{equation}\label{equation-trucation}
    \tilde{D}=D\smallsetminus\left(\bigcup_{i=1}^t U_i\right)
    \end{equation}
    to be the \textit{horocyclic truncation} of $D$ which becomes compact.
    
    For any $z\in\bH^2$, by Proposition \ref{prop-cusp}, each cusp of $D_\Gamma(z)$ is congruent to some $s_i$ in $\Gamma$. We may list its cusps as $s_1(z)=\gamma_{1,z}\cdot s_{i_1},\dots, s_{t_z}(z)=\gamma_{t_z,z}\cdot s_{i_{t_z}}$ for some $\gamma_{1,z},\dots,\gamma_{t_z,z}\in\Gamma$. Note that for each $j$, there are infinitely many $\gamma_{j,z}$ such that $s_j(z)=\gamma_{j,z}s_{i_j}$. But we can restrict the translations to a bounded number by requiring that $\gamma_{j,z}\cdot (D\cap U_{i_j})\subset D_{\gamma}(z)$. Similar to (\ref{equation-horocyclic neighborhood}),
    \[Stab_{\Gamma}(s_j(z))=\gamma_{j,z}Stab_{\Gamma}(s_{i_j})\gamma_{j,z}^{-1}=\{\gamma\in\Gamma\mid \gamma\cdot (\gamma_{j,z}\cdot U_{i_j})\cap (\gamma_{j,z}\cdot U_{i_j})\neq\emptyset\},\]
    and similar to (\ref{equation-trucation}) we define  the horocyclic truncation
    \[\tilde{D}_\Gamma(z)=D_\Gamma(z)\smallsetminus\left(\bigcup_{j=1}^{t_z} \gamma_{j,z}\cdot U_{i_j}\right).\]
    After the above surgery we modify our previous construction of (\ref{equation-construction of geodesic cover}) by
    \begin{equation}\label{equation-construction after surgery}
    \tilde{B}_\Gamma(z_0):=\bigcup_{z\in\tilde{D}}\tilde{D}_\Gamma(z),\ 
    \tilde{U}_\Gamma(z_0):=\{\gamma\in\Gamma\mid \gamma\cdot D\cap \tilde{B}_\Gamma(z_0)\neq\emptyset\}.
    \end{equation}
    The following two facts account for the significance of our construction.
    \begin{lem}\label{lem-compactness}
    	With the above notations, $\tilde{B}_\Gamma(z_0)$ is compact.
    \end{lem}
    \begin{proof}
    	For each $z\in\tilde{D}$ and each cusp $s_j(z)$ of $D_\Gamma(z)$, we have the intersection of two closed sets $\gamma_{j,z}\cdot D\cap D_\Gamma(z)$ contains some horocyclic region $\gamma_{j,z}\cdot U_{i_j}$. Then  $\gamma_{j,z}^{-1}\cdot D\cap\tilde{D}\neq\emptyset$. Since $\tilde{D}$ is compact, the set $\{\gamma_{j,z}\mid z\in\tilde{D}, j\leq t_z\}$ is finite by local finiteness. Thus $\tilde{D}_\Gamma(z)$ is uniformly bounded for $z\in\tilde{D}$ and their union, i.e. $\tilde{B}_\Gamma(z_0)$ is compact.
    \end{proof}
    \begin{lem}\label{lem-modified construction}
    	With previous notations, $U_\Gamma(z_0)\smallsetminus\tilde{U}_\Gamma(z_0)$ is a finite set. 
    \end{lem}
    \begin{proof}
    	Clearly $\tilde{U}_\Gamma(z_0)\subset U_\Gamma(z_0)$ by definition. Suppose $\gamma\in U_\Gamma(z_0)$. Then $\gamma\cdot D\cap D_\Gamma(z)\neq\emptyset$ for some $z\in D$. If $\gamma\notin \tilde{U}_\Gamma(z_0)$, then $\gamma\cdot D\cap\gamma_{j,z} U_{i_j}$ for some horocyclic region $U_{i_j}$ of $D$, or equivalently  $D\cap\gamma^{-1}\gamma_{j,z}U_{i_j}\neq\emptyset$. By Proposition \ref{prop-horocyclic region} and our restriction on the $\gamma_{j,z}$'s, there are finitely many elements, say $\gamma_1,\dots,\gamma_n\in\Gamma$, such that $\gamma^{-1}\gamma_{j,z}=\gamma_k$ for some $k$. Then $\gamma\in\{\gamma_{j,z}\gamma_k^{-1}\mid z\in D, j\leq t_z, k\leq n\}$, a finite set by the last proof. 
    \end{proof}
    To eliminate the restriction of having no elliptic elements, we resort to
    \begin{prop}[Nielsen-Fenchel-Fox]\label{prop-Fox}
    	Any finitely generated Fuchsian group contains a subgroup of finite index without elliptic elements.   	
    \end{prop}
    This was the Fenchel's conjecture proved by Fox \cite{Fox}, the proof of which contains an error later fixed by Chau \cite{Chau}. Note that further by Poincar\'{e}'s lemma in group theory, there exists a normal subgroup of finite index without elliptic elements although we do not need this latter fact.
    
    Together with Proposition \ref{prop-finite index} we are ready to prove the following
    \begin{thm}\label{thm-co-finite}
    	Any co-finite Fuchsian group has finite geodesic covers.
    \end{thm}
    \begin{proof}
	By Proposition \ref{prop-Fox}, choose a finite index subgroup $\Gamma'\leq\Gamma$ which contains no elliptic elements. Fix any $z_0\in\bH^2$, then $\tilde{B}_{\Gamma'}(z_0)$ is compact by Lemma \ref{lem-compactness}, which implies that $\tilde{U}_{\Gamma'}(z_0)$ is finite by local finiteness. Further by Lemma \ref{lem-modified construction} and Proposition \ref{prop-construction of geodesic cover}, $\tilde{U}_{\Gamma'}(z_0)$ is finite geodesic cover of $\Gamma'$. Finally, by Proposition \ref{prop-finite index}, $\Gamma$ has a finite geodesic of size $\leq[\Gamma:\Gamma']|\tilde{U}_{\Gamma'}(z_0)|$.
    \end{proof}
    \subsection{Geodesic cover of geometrically finite groups of second kind}
    For the remaining case where $\Gamma$ is a geometrically finite Fuchsian group of second kind, we modify our construction through a bisection using Nielsen region $N_\Gamma$ introduced as (\ref{equation-Nielsen region}). 
    
    We assume $\Gamma$ is non-elementary with no elliptic elements. Then for each $z\in\bH^2$, the convex fundamental Dirichlet domain $D_\Gamma(z)$ intersects $N_\Gamma$ at a polygon with finitely many sides in $\bH^2$. Moreover, each free side of $D_\Gamma(z)$ is contained in an open interval in $\hat{\R}\smallsetminus\Lambda(\Gamma)$.Write $\hat{D}_\Gamma(z):=D_\Gamma(z)\cap N_\Gamma$. Now for a point $z_0\in\bH^2$, form the set $\hat{B}_\Gamma(z_0):=\cup_{z\in D}\hat{D}_\Gamma(z)$ and construct
    \[\hat{U}_\Gamma(z_0):=\{\gamma\in\Gamma\mid \gamma\cdot D\cap \hat{B}_\Gamma(z_0)\neq\emptyset\}.\]
    Since each $\hat{D}_\Gamma(z)$ is a finite-sided polygon possible with cusps of $\Gamma$, we may perform the same horocyclic surgery as in the last subsection on $\hat{D}_\Gamma(z)$. 
    
    Similar to the proof of Lemma \ref{lem-compactness}, $\hat{B}_\Gamma(z_0)$ has finitely many cusps hence $B_\Gamma(z_0)$ has finitely many free sides, say contained in the intervals $I_1,\dots,I_n\subset\hat{R}\smallsetminus\Lambda(\Gamma)$. Also, similarly by Lemma \ref{lem-compactness} $\hat{U}_\Gamma(z_0)$ is finite. For any $\gamma\in U_{\Gamma}(z_0)\smallsetminus\hat{U}_\Gamma(z_0)$, $\gamma\cdot D$ intersects some of the interval $I_j$. Then it has a free side contained in $I_j$ and must intersect a \textit{hypercyclic region} $U_j$ having same end points (fixed points of some hyperbolic elements of $\Gamma$) with $I_j$. Here a hypercyclic region is congruent to a set of the form $\{z\in\bH^2\mid |\arg(z)-\pi/2|<\theta\}$. However, such $\gamma$ if restricted to having intersection with a Dirichlet region contained in $B_\Gamma(z_0)$ with a free side in $I_j$, amount to be finite due to
    \begin{prop}[Theorem 9.8.2 (iii) of \cite{Beardon}]\label{prop-hypercyclic region}
    	Let $D$ be any locally finite fundamental domain for $\Gamma$ and $U$ be a hypercyclic region. Then $D$ meets some finitely many translates $\gamma\cdot U, \gamma\in\Gamma$.    	
    \end{prop}  
    Now that there are finitely many $I_j$'s, similar to Lemma \ref{lem-modified construction} we see that 
    \[|U_\Gamma(z_0)\smallsetminus\hat{U}_\Gamma(z_0)|<\infty.\]
    Covering the general case of non-elementary geometrically finite Fuchsian groups of second kind using Proposition \ref{prop-Fox} and \ref{prop-finite index}, we have proved
    \begin{thm}\label{thm-second kind}
    	Any non-elementary geometrically finite Fuchsian group of second kind has finite geodesic covers.
    \end{thm}
    Wrapping up the elementary groups resorting to the classification by Proposition \ref{prop-elementary groups} which are easily shown to have finite geodesic covers by hand, we conclude the proof of Theorem \ref{thm-geometrically finite}, i.e. all geometrically finite Fuchsian groups have finite geodesic covers.

\end{document}